\documentclass[a4paper, 11pt]{article}
\usepackage{amsfonts,amsmath,amssymb,amsthm,enumerate,bm,a4}
\usepackage[a4paper,text={6.5in,9in},centering]{geometry}
\usepackage{graphicx}
\usepackage{amsmath}
\usepackage{amsfonts}
\usepackage{mathrsfs}

\usepackage[mathscr]{eucal}
\usepackage{calrsfs}
\usepackage[active]{srcltx}
\usepackage[russian,english]{babel}

\newtheorem{theorem}{Theorem}[section]

\newtheorem{corollary}[theorem]{Corollary}

\theoremstyle{definition}
\newtheorem{remark}[theorem]{Remark}

\newtheorem{definition}[theorem]{Definition}

\numberwithin{equation}{section}
\newcommand{\R}{\mathbb R}

\newcommand{\ent}{\int_{\R}}
\newcommand{\et}{\int_{\R^2}}

\newcommand{\al}{\alpha}
\newcommand{\lf}{\left}
\newcommand{\rh}{\right}

\newcommand{\sH}{\mathscr H}
\newcommand{\sA}{\mathscr A}
\allowdisplaybreaks[4]

\begin{document}
	
\begin{center}
	\Large\bf{ Boundedness of Dunkl-Hausdorff operator in Lebesgue spaces}
	
\end{center}
\vspace{0.1in}

\begin{center}

	{\bf Sandhya Jain}\\
	Department of Mathematics\\
	Vivekananda College (University of Delhi)\\
	Vivek Vihar,  Delhi - 110095, India\\
	(email : singhal.sandhya@gmail.com)
	\medskip
	
	{\bf Alberto Fiorenza}\\
		Dipartimento di Architettura \\
		Universit\`{a} di Napoli Federico II, Via Monteoliveto, 3 \\
		I-80134 Napoli, Italy \\
		and Istituto per le Applicazioni del Calcolo
		``Mauro Picone", sezione di Napoli \\
		Consiglio Nazionale delle Ricerche \\
		via Pietro Castellino, 111 \\
		I-80131 Napoli, Italy
	{email :fiorenza@unina.it}	
	\medskip
	
	{\bf Pankaj Jain}\\
	Department of Mathematics\\
	South Asian University\\
	Akbar Bhawan, Chanakya Puri, New Delhi-110021, India\\
	(email: pankaj.jain@sau.ac.in, pankajkrjain@hotmail.com)
	\medskip
\end{center}
\bigskip

\begin{abstract}
In this paper, the $L^p_v(\R)$-boundedness of the Dunkl-Hausdorff operator $\displaystyle
H_{\al,\phi} f(x)=\ent\frac{ |\phi(t)|}{|t|^{2\al+2}}f\lf(\frac{x}{t}\rh) dt
$
has been characterized and for a certain type of weight $v$, the precise value of the norm $\|H_{\al,\phi}\|_{L^p_v(\R)\to L^p_v(\R)}$ has been obtained. This covers several of the existing results. Analogous results in two dimensions have also been proved.
\end{abstract}

\medskip


\section{Introduction}

Let $v$ be a weight function, i.e., a function which is measurable, positive and finite almost everywhere on the specified domain. By $L^p_v(\R)$, $1\le p<\infty$, we denote the weighted Lebesgue space and a norm of a function $f\in L^p_v(\R)$ is given by
$$
\|f\|_{L^p_v(\R)}:=\lf(\ent|f(x)|^p v(x) dx\rh)^{1/p}.
$$
Occasionally, we shall be referring to the specific weight $v(x)=|x|^{2\al +1}$. The corresponding weighted Lebesgue space will be denoted by $L^p_\al(\R)$. The non-weighted Lebesgue space, i.e., when $v\equiv 1$, will be denoted by $L^p(\R)$.

Let $\phi\in L^1(\R)$. In the present paper, we are concerned with the Dunkl-Hausdorff operator \cite{daher, daher1, daher2, dunkl}
$$
H_{\al,\phi} f(x)=\ent\frac{ |\phi(t)|}{|t|^{2\al+2}}f\lf(\frac{x}{t}\rh) dt.
$$
When $\al=-1/2$, The operator $H_{\al,\phi}$ is the famous Hausdorff operator
$$
H_\phi f(x)=\ent\frac{|\phi(t)|}{|t|}f\lf(\frac{x}{t}\rh) dt,
$$
from which several well known operators can be deduced for suitable choices of $\phi$, e.g., for $\phi(t)=\frac{1}{t}\chi_{(1,\infty)}(t)$, the operator $H_\phi$ reduces to the standard Hardy averaging operator $$Hf(x)=\frac{1}{x}\int_0^x f(t)\,dt$$
while for $\phi(t)=\chi_{[0,1]}(t)$, it reduces to the adjoint of Hardy averaging operator $$H^*f(x)=\int_x^\infty \frac{f(t)}{t}\,dt.$$
Similarly, other operators like Calderon operator, Ces\'aro operator and fractional Riemann Liouville operator can also be deduced from $H_\phi$, see \cite{gorka, jain} for details.
For more updates on the Hausdorff operator, its extensions and in the framework of other function spaces one may refer to \cite{ar, chen, kan, miya, lm} and the survey \cite{L}.

By the replacement $\phi(s)=\frac{1}{s}\psi\lf(\frac{1}{s}\rh)$, $s>0$,  the operator $H_\phi$ (considered on $\R^+$) becomes equivalent to
$$G_\psi g(x)=\frac{1}{x}\int_0^\infty \psi\lf(\frac{t}{x}\rh)g(t)\,dt.$$
It was proved by Golberg (\cite{gold}, Theorem 1) that if $\psi\ge 0$ on $\R^+$ is such that $\displaystyle\ent \frac{\psi(t)}{\sqrt{t}} dt =:K<\infty$, then the operator $G_\psi$ (and consequently $H_\phi$) is a bounded operator on $L^2(\R^+)$ and $\|G_\psi\|\leq K.$ The $L^p$-boundedness of $G_\psi$ is derived in (\cite{HLP}, Theorem 319) and for many other extensions with sharp constants one may refer to  (\cite{KS}, Theorem 6.4 and bibliographic notes to Chapter 2 therein). In \cite{jain}, the authors reestablished the $L^p(\R^+)$-boundedness of $G_\psi$ and via a new proof of the lower bound, obtained the precise value of $\|G_\psi\|_{L^p(\R^+)\to L^p(\R^+)}$ as
$$
\|G_\psi\|_{L^p(\R^+)\to L^p(\R^+)}= \int_0^\infty \frac{\psi(t)}{{t^{1/p}}}\, dt =:K_p,\quad 1<p<\infty.
$$

Recently, in \cite{gorka}, a two weight characterization of the boundedness of $H_\phi$ between $L^p_v(\R^+)$ and $L^p_w(\R^+)$ has been given. Moreover, in the same paper, the corresponding boundedness has been studied in the framework of other function spaces as well, namely, grand Lebesgue spaces and variable exponent Lebesgue spaces.

Coming back to the Dunkl-Hausdorff operator $H_{\al,\phi}$, its $L_\al^1(\R)$ boundedness has been proved in \cite{daher} whereas $L_\al^p(\R)$ boundedness is obtained in \cite{daher2} and in each case, a sufficient condition has been provided. We, in this paper, generalize these results by providing a characterization for the $L_v^p(\R)$ boundedness of $H_{\al,\phi}$ and for a certain type of weight $v$, we provide the precise value of the norm $\|H_{\al,\phi}\|_{L^p_v(\R)\to L^p_v(\R)}$. Moreover, a sufficient condition has been proved for two weights and two indices boundedness, i.e., $H_{\al,\phi}:L^p_w(\R)\to L^q_v(\R)$ boundedness. These results have also been proved in the two dimensional framework.

\section{One dimensional case}

We begin by proving the following:

\begin{theorem}\label{t2.1}
	Let $1<p<\infty$, $v$ be weight function and $\phi\in L^1(\R)$ be such that $$\displaystyle \ent \frac{|\phi(t)|}{{|t|^{{2\al+2}-\frac{1}{p}}}}\lf(\sup_{y\in\R}\frac{v(ty)}{v(y)}\rh)^{1/p} dt =:A_{\sup}<\infty.$$ Then the operator $H_{\al,\phi}$ is a bounded operator on $L^p_v(\R)$ and 
	$$\|H_{ \al,\phi}f\|_{L^p_v(\R)}\leq A_{\sup} \|f\|_{L^{p}_{v}(\R)}.
	$$
\end{theorem}
\begin{proof}
	If $f \in L^p_v(\R)$ then by using generalised Minkowski inequality, change of variables and H\"older's inequality, we have
	\begin{align}\label{enew1}
	\|H_{ \al,\phi}f\|_{L^p_v(\R)}
	& =\lf(\ent |H_{ \al,\phi} f(x)|^p v(x)dx \rh)^\frac{1}{p} \nonumber\\
	& = \lf(\ent\lf|\ent \frac{|\phi(t)|}{|t|^{2\al+2}}f\lf(\frac{x}{t}\rh) dt\rh|^p v(x)dx \rh)^\frac{1}{p} \nonumber\\
	& \leq \lf(\ent \lf(\ent \frac{|\phi(t)|}{|t|^{2\al+2}}\lf|f\lf(\frac{x}{t}\rh)\rh|v^{\frac{1}{p}}(x)dt\rh)^p dx \rh)^\frac{1}{p}\nonumber\\
	& \leq  \ent \frac{|\phi(t)|}{|t|^{2\al+2}}|t|^{\frac{1}{p}} \lf(\ent \lf|f(y)\rh|^p v(yt)dy\rh)^\frac{1}{p} dt\nonumber\\
	& =  \ent \frac{|\phi(t)|}{|t|^{{2\al+2}-\frac{1}{p}}}\lf(\ent \lf|f(y)\rh|^p v(y)v^{-1}(y) v(yt)dy\rh)^{\frac{1}{p}} dt\\
	&\leq \lf(\ent \frac{|\phi(t)|}{|t|^{{2\al+2}-\frac{1}{p}}} \lf(\sup_{y\in\R}\frac{v(ty)}{v(y)}\rh)^{\frac{1}{p}} dt \rh)\lf(\ent |f(y)|^p v(y) dy \rh)^{\frac{1}{p}}\nonumber\\
	& = A_{\sup} \|f\|_{L^p_v(\R)}\nonumber
	\end{align}
and the assertion follows.
\end{proof}

The following theorem provides a converse of Theorem \ref{t2.1}. Here and throughout $p'$ denotes the conjugate index to $p$, i.e., ${1\over p}+{1\over p'}=1.$

	\begin{theorem}\label{t2.2}
		Let $1<p<\infty$, $v$ be a weight function and $\phi\in L^1(\R)$. If the operator $H_{\al,\phi}$ is a bounded operator on $L^p_v(\R)$,  then $$\|H_{ \al,\phi}\|\geq \ent \frac{|\phi(t)|}{{|t|^{2\al+2-\frac{1}{p}}}}\lf(\inf_{y\in\R}\frac{v(ty)}{v(y)}\rh)^{1/p} dt=:A_{\inf}.$$
	\end{theorem}
	
\begin{proof}
Let us consider $0\le f\in L^p_v(\R)$ and $0\le g\in L^{p'}_{v^{1-p'}}(\R)$. On using Fubini's Theorem and H\"older's inequality, We have
	\begin{align}\label{K2}
	{J} & := \ent \frac{|\phi(t)|}{|t|^{2\al+2}}\lf( \ent f\lf(\frac{x}{t}\rh) g(x)dx\rh) dt\\
	& = \ent g(x)\lf(\ent \frac{|\phi(t)|}{|t|^{2\al+2}}f\lf(\frac{x}{t}\rh)dt\rh) dx\nonumber\\
	&\leq \ent |g(x)| \lf|(H_{\al,\phi} f )(x)\rh| dx \nonumber\\
	&= \ent |g(x)| v^{\frac{1}{p}}(x)v^{\frac{-1}{p}}(x)\lf|(H_{\al,\phi} f )(x)\rh|  dx \nonumber\\
	& \leq \lf(\ent \lf|\lf(H_{\al,\phi}f\rh)(x)\rh|^p v(x)  dx\rh)^{1/p}\lf(\ent |g(x)|^{p'} v^{\frac{-p'}{p}}(x)dx\rh)^{^{\frac{1}{p'}}}\nonumber \\
	& \leq \|H_{\al,\phi}\| \|f\|_{L^p_v(\R)}\|g\|_{L^{p'}_{v^{1-p'}}(\R)}.\label{h2.1}
	\end{align}
	For any interval $I=(a,b)$, $I'$ will denote the interval $(-b,-a)$. Now, for $u\in (0,1)$, let $I_1 = (u, 1/u)$ so that $I_1'=(-1/u,-u)$. Define the test functions
	$$f_u(x)=\frac{v^{-1/p}(x)}{|x|^{1/p}}\chi_{I_1\cup I'_1}(x), \qquad g_u(x)=\frac{v^{1/p}(x)}{|x|^{1/p'}}\chi_{I_1\cup I'_1}(x).$$
	Then it can be calculated that 
	\begin{equation}\label{h2.2}
	\|f_u\|_{L^p_v(\R)}^p=\|g_u\|_{L^{p'}_{v^{1-p'}}(\R)}^{p'}=4 \log(1/u).
	\end{equation}
	Also, we have 
	\begin{align}\label{2.2}
	h_u(t)
	&:= \ent f_u\lf(\frac{x}{t}\rh) g_u(x)dx\nonumber\\
	& = |t|^{1/p} \ent \frac{1}{|y|}\chi_{I_1\cup I'_1}(y)\chi_{I_1\cup I'_1}\lf(ty\rh) v^{-1/p}(y)v^{1/p}(ty) dy\nonumber\\
	&= |t|^{1/p}\ent \frac{1}{|y|}\chi_{I_{t,u}}(y)v^{-1/p}(y)v^{1/p}(ty) dy\nonumber\\
 &\geq \inf_{y \in \R} \lf(\frac{v(ty)}{v(y)}\rh)^{1/p} |t|^{1/p}\ent \frac{1}{|y|}\chi_{I_{t,u}}(y) dy
	\end{align}
	where
\begin{align*}
I_{t,u} &=
\begin{cases}
\lf(I_1\cup I'_1\rh)\cap(I_2\cup I'_2),& {\rm if}\, t\geq 0\\
\lf(I_1\cup I'_1\rh)\cap(I_3\cup I'_3), & {\rm if}\, t< 0
\end{cases}\\
&=
\begin{cases}
\lf(I_1\cap {I_2}\rh)\cup(I'_1\cap I'_2),& {\rm if}\, t\geq 0\\
\lf(I_1\cap I'_3\rh)\cup(I'_1\cap I_3), & {\rm if}\, t< 0
\end{cases}
\end{align*}
with $I_2=(\frac{u}{t}, \frac{1}{ut})$ and $I_3=(\frac{1}{tu}, \frac{u}{t})$. We divide $\R$ as
	$$\R=\lf(-\infty, -\frac{1}{u^2},\rh]\cup \lf(-\frac{1}{u^2}, -1\rh]\cup(-1, -u^2]\cup(-u^2, 0]\cup(0,u^2]\cup(u^2,1]\cup\lf(1,\frac{1}{u^2}\rh]\cup\lf(\frac{1}{u^2},\infty\rh).$$
	If $t\in \lf(-\infty, -\frac{1}{u^2},\rh]\cup[-u^2,u^2]\cup\lf[\frac{1}{u^2},\infty\rh)\cup\{-1,1\}$, then $I_{t,u}=\emptyset$, so that in this case
	\begin{equation}\label{h2.3}
	h_u(t)=0.
	\end{equation}
	If $t\in (-1/u^2,-1)$, then $I_{t,u} = \lf(u,-\frac{1}{tu}\rh)\cup\lf(\frac{1}{ut},{-u}\rh)$ and 
	\begin{equation}\label{e2.1}
	\ent \frac{1}{|y|}\chi_{I_{t,u}}dy=-2\log|t|+4\log \frac{1}{u} .
	\end{equation}
	If $t\in \lf(-1, -u^2\rh)$, then $I_{t,u} = \lf(-\frac{1}{u}, \frac{u}{t}\rh)\cup\lf(-\frac{u}{t}, \frac{1}{u}\rh)$ and we have
	\begin{equation}\label{e2.2}
	\ent \frac{1}{|y|}\chi_{I_{t,u}}dy=2\log|t|+4\log \frac{1}{u}.
	\end{equation}
	If $t\in (u^2,1)$, then $I_{t,u}=\lf(\frac{u}{t}, \frac{1}{u}\rh)\cup\lf(-\frac{1}{u}, -\frac{u}{t}\rh)$ and 
	\begin{equation}\label{h2.4}
	\ent \frac{1}{|y|}\chi_{I_{t,u}}dy=2\log|t|+4\log \frac{1}{u}.
	\end{equation}
	If $t\in \lf(1,\frac{1}{u^2}\rh)$, then $I_{t,u}= \lf(u, \frac{1}{tu}\rh)\cup\lf(-\frac{1}{tu}, -u\rh)$ and we have
	\begin{equation}\label{h2.5}
	\ent \frac{1}{|y|}\chi_{I_{t,u}}dy=-2\log|t|+4\log \frac{1}{u}.
	\end{equation}
	On taking $f$ and $g$ as $f_u$ and $g_u$ in (\ref{K2}) and then using \eqref{2.2}-\eqref{h2.5}, we obtain

	\begin{align}\label{h2.6}
{J} &= \lf(\int^{-1/u^2}_{-\infty}+\int_{-1/u^2}^{-1}+\int^{-u^2}_{-1}+\int^0_{-u^2}+\int_0^{u^2}+\int_{u^2}^1+\int_1^{1/u^2}+\int_{1/u^2}^\infty\rh) \frac{|\phi(t)|}{|t|^{2\al+2}} h_u(t) dt \nonumber\\
	& \geq 4 \log \frac{1}{u} \lf[\int^{-u^2}_{-1/u^2}\inf_{y \in \R} \lf(\frac{v(ty)}{v(y)}\rh)^{\frac{1}{p}}\frac{|\phi(t)|}{|t|^{{2\al+2}-\frac{1}{p}}}\lf(1-\frac{\xi(t)}{4\log \frac{1}{u}}\rh)dt\rh.\nonumber\\
	& \hskip 1.0 in \lf.+ \int^{1/u^2}_{u^2}\inf_{y \in \R} \lf(\frac{v(ty)}{v(y)}\rh)^{\frac{1}{p}}\frac{|\phi(t)|}{|t|^{{2\al+2}-\frac{1}{p}}}\lf(1-\frac{\xi(t)}{4\log \frac{1}{u}}\rh)dt\rh],
	\end{align}
			where 
	\begin{equation*}
	\xi(t)=
	\begin{cases}
	 2\log |t|,& {\rm if}\, t\in (-1/u^2, -1)\cup (1,1/u^2) \\
-2\log |t| , & {\rm if}\, t \in (-1, -u^2)\cup (u^2, 1).	
	\end{cases}	
	\end{equation*}
	Now, using \eqref{h2.6} and \eqref{h2.2} in \eqref{h2.1} for $f=f_u$ and $g=g_u$, we get
	\begin{align*}
	&\int^{-u^2}_{-1/u^2}\inf_{y \in \R} \lf(\frac{v(ty)}{v(y)}\rh)^{\frac{1}{p}}\frac{|\phi(t)|}{|t|^{{2\al+2}-\frac{1}{p}}}\lf(1-\frac{\xi(t)}{4\log \frac{1}{u}}\rh)dt \\
	& \qquad + \int^{1/u^2}_{u^2}\inf_{y \in \R} \lf(\frac{v(ty)}{v(y)}\rh)^{\frac{1}{p}}\frac{|\phi(t)|}{|t|^{{2\al+2}-\frac{1}{p}}}\lf(1-\frac{\xi(t)}{4\log \frac{1}{u}}\rh)dt \leq \|H_{\al,\phi}\|.
	\end{align*}
	By the Monotone Convergence Theorem, the LHS $\displaystyle \uparrow A_{\inf}$ as $u\rightarrow 0$ and we are done.
\end{proof}

In view of Theorems \ref{t2.1} and \ref{t2.2}, a characterization for the boundedness of $H_{\al,\phi}:L^p_v(\R)\to L^p_v(\R)$ can be derived. In fact, the following is immediate:

\begin{theorem}\label{t2.3}
Let $1<p<\infty$, $v$ be weight function and $\phi\in L^1(\R)$. Let the following be satisfied for some constant $c>0$: 
$$
\sup_{y\in\R}\frac{v(ty)}{v(y)}\le c \inf_{y\in\R}\frac{v(ty)}{v(y)}.
$$ 
Then the operator $H_{\al,\phi}:L^p_v(\R)\to L^p_v(\R)$ is  bounded if and only if $A_{\sup} <\infty$ and moreover the following estimates hold:
\begin{equation*}
{1\over c} A_{\sup} \le \|H_{\al,\phi}\|_{L^p_v(\R)\to L^p_v(\R)}\le A_{\sup}.
\end{equation*}
\end{theorem} 

\begin{corollary}\label{t2.4} Let $1<p<\infty$, $v$ be weight function and $\phi\in L^1(\R)$. If there exists a function $h$ such that $v(xy)=v(x)h(y)$, then $A_{\sup} = A_{\inf}$ and
\begin{equation*}
\|H_{\al,\phi}\|_{L^p_v(\R)\to L^p_v(\R)} = \ent \frac{|\phi(t)|}{{|t|^{2\al+2-\frac{1}{p}}}}h^{1/p}(t)\, dt.
\end{equation*}
	
\end{corollary}

For $\al=-1/2$ and $\phi(t)={1\over t}\chi_{(1,\infty)}(t)$, as pointed out earlier, the operator $H_{\al,\phi}$ becomes the Hardy averaging operator
$$
Hf(x)={1\over x}\int_0^x f(t)\, dt.
$$
Further, if we take $v(t)=t^\beta, \, \beta<p-1$, then
$\displaystyle A_{\sup}=\frac{p}{\beta-p-1}$ and
$$
\|H\|_{L^p_{t^\beta}(\R)\to L^p_{t^\beta}(\R)} = \frac{p}{\beta-p+1}.
$$
The above discussion leads to the following corollary which, in fact, is the classical Hardy inequality (see e.g. \cite[Theorem 330]{HLP}, \cite[(3.6) p. 23]{KMP} or \cite[Theorem 6 p. 726]{KMP2}):

\begin{corollary}
	Let $1<p<\infty$ and $\beta<p-1$ be a the weight function. Then the inequality
$$
\|Hf\|_{L^p_{x^\beta}(\R)} \le \left(\frac{p}{\beta-p+1}  \right) \|f\|_{L^p_{x^\beta}(\R)}
$$
holds and the constant $ \left(\frac{p}{\beta-p+1}  \right)$ is sharp.
\end{corollary}

\begin{remark}
	\begin{itemize} 
	\item [(i)] When $p=1$ and $v(x)=|x|^{2\al +1}$, Theorem \ref{t2.1} reduces to (\cite{daher}, Theorem 3.1).
	
	\item [(ii)] When $v(x)
=|x|^{2\al +1}$, Theorem \ref{t2.1} reduces to (\cite{daher2}, Theorem 1)	

\item [(iii)] When $\al=-1/2$, Theorems \ref{t2.1}, \ref{t2.2}, \ref{t2.3} and Corollary \ref{t2.4} reduce to (\cite{gorka}, Theorems 1(i), 1(ii), Corollaries 1 and 2, respectively for $w=v$). Moreover, the functions here are defined on $\R$ unlike in \cite{gorka} where the domain is $\R^+$.
\end{itemize}
	
\end{remark}


\section{Two dimensional case}

In this section, we derive the two dimensional analogues of the results proved in Section 2.

\begin{definition}
For $\phi\in L^1(\R^2)$, the two dimensional Dunkl-Hausdorff operator  is defined by
$$\sH_{\al, \phi}(f(x_1,x_2)) = \et \frac{|\phi(t_1,t_2)|}{|t_1t_2|^{2\al+2}}f\lf(\frac{x_1}{t_1},\frac{x_2}{t_2}\rh)dt_1dt_2.$$
\end{definition}

Now, we prove the boundedness of $\sH_{\al, \phi}$. By using generalised Minkowski inequality, change of variables and H\"older's inequality in two dimensions, the following theorem can be proved along the same lines as in Theorem \ref{t2.1}.

\begin{theorem}\label{t3.2}
	Let $1<p<\infty$, $v$ be weight function and $\phi\in L^1(\R^2)$ be such that $$\displaystyle \et \frac{|\phi(t_1, t_2)|}{{|t_1t_2|^{2\al+2-\frac{1}{p}}}}\lf(\sup_{(y_1, y_2)\in\R^2}\frac{v(t_1y_1, t_2y_2)}{v(y_1, y_2)}\rh)^{1/p} dt_1dt_2 =:\sA_{\sup}<\infty.$$ Then the operator $\sH_{\al,\phi}:L^p_v(\R^2)\to L^p_v(\R^2)$ is bounded and 
	$$
	\|\sH_{\al, \phi}f\|_{L^p_v(\R^2)}\leq \sA_{\sup} \|f\|_{L^{p}_{v}(\R^2)}.
	$$
\end{theorem}

Towards the converse of Theorem \ref{t3.2}, we prove the following:
\begin{theorem}\label{t3.3}
	Let $1<p<\infty$, $v$ be weight function and $\phi\in L^1(\R^2)$. If the operator $\sH_{\al,\phi}:L^p_v(\R^2)\to L^p_v(\R^2)$ is bounded, then $$\|\sH_{\al, \phi}\|\geq \et \frac{|\phi(t_1, t_2)|}{{|t_1t_2|^{2\al+2-\frac{1}{p}}}}\lf(\inf_{(y_1 y_2)\in\R^2}\frac{v(t_1y_1, t_2y_2)}{v(y_1,y_2)}\rh)^{1/p} dt_1dt_2=:\sA_{\inf}.$$
\end{theorem}

\begin{proof}
Let $0\le f\in L^p_v(\R^2)$ and $0\le g\in L^{p'}_{v^{1-p'}}(\R^2)$. On using Fubini's Theorem and H\"older's inequality, We have
	\begin{align}\label{h0}
	\mathcal{J}& := \et \frac{|\phi(t_1, t_2)|}{|t_1t_2|^{2\al+2}} \et f\lf(\frac{x_1}{t_1},\frac{x_2}{t_2}\rh) g(x_1,x_2)dx_1dx_2 dt_1dt_2\\
	& = \et g(x_1,x_2) \lf(\et \frac{|\phi(t_1, t_2)|}{|t_1t_2|^{2\al+2}}f\lf(\frac{x_1}{t_1},\frac{x_2}{t_2}\rh)dt_1dt_2\rh) dx_1dx_2\nonumber\\
	&\leq \et |g(x_1,x_2)| \lf|\lf(\sH_{\al, \phi} f \rh)(x_1,x_2)\rh|  dx_1dx_2\nonumber\\
	& \leq \lf(\et \lf|\lf(\sH_{\al, \phi}f\rh)(x_1,x_2)\rh|^p v(x_1, x_2)  dx_1dx_2\rh)^{1/p}\lf(\ent |g(x_1,x_2)|^{p'} v(x_1, x_2) ^{\frac{-p'}{p}} dx_1dx_2\rh)^{1/p'}\nonumber \\
	& \leq \|\sH_{\al, \phi}\| \|f\|_{L^p_v(\R^2)}\|g\|_{L^{p'}_{v^{1-p'}}(\R^2)}.\label{h2.1.1}
	\end{align}
	Now, for $u\in (0,1)$, we define the test functions
	\begin{align*}
	f_u(x_1,x_2)&=\frac{v^{-1/p}(x_1,x_2)}{|x_1x_2|^{\frac{1}{p}}}\chi_{I_1\cup I'_1\times I_1\cup I'_1}(x_1,x_2),\\
	 g_u(x_1,x_2)&=\frac{v^{1/p}(x_1,x_2)}{|x_1x_2|^{\frac{1}{p'}}}\chi_{I_1\cup I'_1\times I_1\cup I'_1}(x_1,x_2),
	 \end{align*}
	where $I_1 = (u, 1/u)$ and as before $I'_1 = (-1/u, -u)$. Then it can be calculated that 
	\begin{equation}\label{h2.1.2}
	\|f_u\|_{{L^p_v}(\R^2)}^p=\|g_u\|_{L^{p'}_{{v^{1-p'}}}(\R^2)}^{p'}=(4 \log(1/u))^2.
	\end{equation}
	Also, on taking ${x_i}/{t_i}= y_i$ for $ i = 1,2$, we have 
	\begin{align*}
	h_u(t_1,t_2)
	&:= \et f_u\lf(\frac{x_1}{t_1}, \frac{x_2}{t_2}\rh) g_u(x_1,x_2)dx_1dx_2\\
	& = |t_1t_2|^{\frac{1}{p}} \et \frac{1}{|y_1y_2|}\chi_{I_1\cup I'_1\times I_1\cup I'_1}(y_1,y_2)v^{-\frac{1}{p}}(y_1,y_2)v^{\frac{1}{p}}(t_1y_1,t_2y_2)\chi_{I_1\cup I'_1\times I_1\cup I'_1}(t_1y_1,t_2y_2)dy_1dy_2\\
	&\geq \inf_{(y_1,y_2)\in\R^2}\frac{v^{\frac{1}{p}}(t_1y_1,t_2y_2)}{v^{\frac{1}{p}}(y_1,y_2)} |t_1t_2|^{\frac{1}{p}}\et \frac{1}{|y_1y_2|}\chi_{I_{t_1,t_2,u}}(y_1,y_2)dy_1dy_2,
	\end{align*}
	where 
		\begin{align*}
	I_{t_1,t_2,u}&= 
	\begin{cases}
	 \lf(I_1\cup I'_1\rh) \cap\lf(I_2\cup I'_2\rh)\times \lf(I_1\cup I'_1\rh) \cap\lf(I_3\cup I'_3\rh), & {\rm if}\, t_1, t_2>0\\
\lf(I_1\cup I'_1\rh) \cap\lf(I_4\cup I'_4\rh)\times \lf(I_1\cup I'_1\rh) \cap\lf(I_5\cup I'_5\rh) , & {\rm if}\,  t_1, t_2<0\\
	\lf(I_1\cup I'_1\rh) \cap\lf(I_2\cup I'_2\rh)\times \lf(I_1\cup I'_1\rh) \cap\lf(I_5\cup I'_5\rh) , & {\rm if}\, t_1>0, t_2<0\\
	\lf(I_1\cup I'_1\rh) \cap\lf(I_4\cup I'_4\rh)\times \lf(I_1\cup I'_1\rh) \cap\lf(I_3\cup I'_3\rh) , & {\rm if}\, t_1<0, t_2>0.
	\end{cases}	\\	
&= 
	\begin{cases}
	 \lf(I_1\cap {I}_2\rh) \cup\lf(I'_1\cap I'_2\rh)\times \lf(I_1\cap I_3\rh) \cup\lf(I'_1\cap I'_3\rh), & {\rm if}\, t_1, t_2>0\\
\lf(I_1\cap I'_4\rh) \cup\lf(I'_1\cap I_4\rh)\times \lf(I_1\cap I'_5\rh) \cup\lf(I'_1\cap {I}_5\rh) , & {\rm if}\,  t_1, t_2<0\\
	\lf(I_1\cap {I}_2\rh) \cup\lf(I'_1\cap I'_2\rh)\times \lf(I_1\cap I'_5\rh) \cup\lf(I'_1\cap {I}_5\rh)  , & {\rm if}\, t_1>0, t_2<0\\
	\lf(I_1\cap I'_4\rh) \cup\lf(I'_1\cap I_4\rh)\times \lf(I_1\cap I_3\rh) \cup\lf(I'_1\cap I'_3\rh) , & {\rm if}\, t_1<0, t_2>0,
	\end{cases}	
	\end{align*}
with	
	$$
	I_2= \lf(\frac{u}{t_1}, \frac{1}{u t_1}\rh)\qquad I_3= \lf(\frac{u}{t_2}, \frac{1}{u t_2}\rh)\qquad I_4= \lf(\frac{1}{u t_1},\frac{u}{t_1}\rh)\qquad I_5=\lf(\frac{1}{u t_2}, \frac{u}{t_2}\rh).
	$$
It is observed that if $t_1\in (-\infty, -\frac{1}{u^2}]\cup [-u^2,u^2]\cup\lf[\frac{1}{u^2},\infty\rh)\cup\{-1, 1\}$ and $t_2\in (-\infty, \infty)$, then 
	$I_{t_1,t_2,u} = \emptyset$
	and therefore, in this case $h_u(t_1,t_2)=0$. The same is the situation if $t_2\in(-\infty, -\frac{1}{u^2}]\cup [-u^2,u^2]\cup\lf[\frac{1}{u^2},\infty\rh)\cup\{-1, 1\}$ and $t_1\in (-\infty, \infty)$, then  $h_u(t_1,t_2)=0$.
	We deal with the remaining cases as follows:
\smallskip

	\noindent Case 1 : $t_1,t_2\in(u^2,1)$. In this case, it can be worked out that 
	\begin{equation*}
	I_{t_1,t_2,u} = I_6\cup I'_6 \times I_7\cup I'_7, \quad \text{where } I_6= \lf(\frac{u}{t_1}, \frac{1}{u}\rh ),\quad I_7= \lf(\frac{u}{t_2}, \frac{1}{u}\rh )
	\end{equation*}
	and therefore,
	\begin{equation*}
	\et \frac{1}{|y_1y_2|}\chi_{I_{t_1,t_2,u}} (y_1,y_2)dy_1dy_2=\lf(4\log \frac{1}{u}\rh)^2\lf(1-\frac{\log |\frac{1}{t_1}|}{2\log \frac{1}{u}}\rh)\lf(1-\frac{\log |\frac{1}{t_2}|}{2\log \frac{1}{u}}\rh).
	\end{equation*}

	\noindent Case 2 : $t_1\in (1,\frac{1}{u^2})$, $t_2\in (u^2,1)$. In this case
	\begin{equation*}
	I_{t_1,t_2,u} = I_8\cup I'_8\times I_7\cup I'_7, \qquad \text{where } I_8=\lf(u, \frac{1}{ut_1}\rh)
	\end{equation*}
	so that
	\begin{equation*}
	\et \frac{1}{|y_1y_2|}\chi_{I_{t_1,t_2,u}}(y_1,y_2)dy_1dy_2=\lf(4\log \frac{1}{u}\rh)^2\lf(1-\frac{\log |t_1|}{2\log \frac{1}{u}}\rh)\lf(1-\frac{\log |\frac{1}{t_2}|}{2\log \frac{1}{u}}\rh).
	\end{equation*}

	\noindent Case 3 : $t_1\in (u^2,1)$, $t_2\in (1,\frac{1}{u^2})$. In this case
	\begin{equation*}
		I_{t_1,t_2,u} = I_6\cup I'_6 \times I_9\cup I'_9, \qquad \text{where } I_9=\lf(u, \frac{1}{ut_2}\rh)
\end{equation*}
	so that
	\begin{equation*}
	\et \frac{1}{|y_1y_2|}\chi_{I_{t_1,t_2,u}}(y_1,y_2)dy_1dy_2=\lf(4\log \frac{1}{u}\rh)^2\lf(1-\frac{\log |\frac{1}{t_1}|}{2\log \frac{1}{u}}\rh)\lf(1-\frac{\log |t_2|}{2\log \frac{1}{u}}\rh).
	\end{equation*}

	\noindent Case 4 : $t_1,t_2\in (1,\frac{1}{u^2})$. In this case 
	\begin{equation*}
	I_{t_1,t_2,u} = I_8\cup I'_8  \times I_9\cup I'_9
	\end{equation*}
	so that
	\begin{equation*}
	\et \frac{1}{|y_1y_2|}\chi_{I_{t_1,t_2,u}} (y_1,y_2)dy_1dy_2=\lf(4\log \frac{1}{u}\rh)^2\lf(1-\frac{\log |t_1|}{2\log \frac{1}{u}}\rh)\lf(1-\frac{\log |t_2|}{2\log \frac{1}{u}}\rh).
	\end{equation*}
	
	\noindent Case 5 : $t_1\in (-\frac{1}{u^2}, -1) ,t_2\in(u^2,1)$. In this case, it can be worked out that 
	\begin{equation*}
	I_{t_1,t_2,u} = I_{10}\cup I'_{10} \times I_7\cup I'_7, \qquad \text{where } I_{10}= \lf(u, -\frac{1}{ut_1}\rh )
	\end{equation*}
	and therefore,
	\begin{equation*}
	\et \frac{1}{|y_1y_2|}\chi_{I_{t_1,t_2,u}}(y_1,y_2)dy_1dy_2=\lf(4\log \frac{1}{u}\rh)^2\lf(1-\frac{\log |t_1|}{2\log \frac{1}{u}}\rh)\lf(1-\frac{\log |\frac{1}{t_2}|}{2\log \frac{1}{u}}\rh).
	\end{equation*}
	
	\noindent Case 6 : $t_1\in (-1,-u^2)$, $t_2\in (u^2,1)$. In this case
	\begin{equation*}
	I_{t_1,t_2,u}=I_{11}\cup I'_{11} \times I_7\cup I'_7, \qquad \text{where } I_{11}=\lf(-\frac{u}{t_1}, \frac{1}{u}\rh)
	\end{equation*}
	so that
	\begin{equation*}
\et \frac{1}{|y_1y_2|}\chi_{I_{t_1,t_2,u}}(y_1,y_2)dy_1dy_2=\lf(4\log \frac{1}{u}\rh)^2\lf(1-\frac{\log |\frac{1}{t_1}|}{2\log \frac{1}{u}}\rh)\lf(1-\frac{\log |\frac{1}{t_2}|}{2\log \frac{1}{u}}\rh).
	\end{equation*}
	
	\noindent Case 7 : $t_1\in (-\frac{1}{u^2},-1)$, $t_2\in (1,\frac{1}{u^2})$. In this case
	\begin{equation*}
		I_{t_1,t_2,u} = I_{10}\cup I'_{10} \times I_9\cup I'_9 
\end{equation*}
	so that
	\begin{equation*}
\et \frac{1}{|y_1y_2|}\chi_{I_{t_1,t_2,u}} (y_1,y_2)dy_1dy_2=\lf(4\log \frac{1}{u}\rh)^2\lf(1-\frac{\log |t_1|}{2\log \frac{1}{u}}\rh)\lf(1-\frac{\log |t_2|}{2\log \frac{1}{u}}\rh).
	\end{equation*}

	\noindent Case 8 : $t_1\in(-1, -u^2)$, $t_2\in (1,\frac{1}{u^2})$. In this case 
	\begin{equation*}
	I_{t_1,t_2,u} = I_{11}\cup I'_{11}  \times I_9\cup I'_9
	\end{equation*}
	so that
	\begin{equation*}
	\et \frac{1}{|y_1y_2|}\chi_{I_{t_1,t_2,u}}(y_1,y_2)dy_1dy_2=\lf(4\log \frac{1}{u}\rh)^2\lf(1-\frac{\log |\frac{1}{t_1}|}{2\log \frac{1}{u}}\rh)\lf(1-\frac{\log |t_2|}{2\log \frac{1}{u}}\rh).
	\end{equation*}

	\noindent Case 9 : $t_1\in (-\frac{1}{u^2}, -1)$, $t_2\in (-1, -u^2)$. In this case
	\begin{equation*}
	I_{t_1,t_2,u} = I_{10}\cup I'_{10} \times I_{12}\cup I'_{12}, \qquad \text{where } I_{12}= \lf( -\frac{u}{t_2}, \frac{1}{u}\rh )
	\end{equation*}
	so that
	\begin{equation*}
	\et \frac{1}{|y_1y_2|}\chi_{I_{t_1,t_2,u}}(y_1,y_2)dy_1dy_2=\lf(4\log \frac{1}{u}\rh)^2\lf(1-\frac{\log |t_1|}{2\log \frac{1}{u}}\rh)\lf(1-\frac{\log |\frac{1}{t_2}|}{2\log \frac{1}{u}}\rh).
	\end{equation*}
	
	\noindent Case 10 : $t_1, t_2\in (-\frac{1}{u^2}, -1)$. In this case, it can be worked out that 
	\begin{equation*}
	I_{t_1,t_2,u} = I_{10}\cup I'_{10} \times I_{13}\cup I'_{13}, \qquad \text{where } I_{13}= \lf( u, -\frac{1}{ut_2}\rh )
	\end{equation*}
	and therefore,
	\begin{equation*}
	\et \frac{1}{|y_1y_2|}\chi_{I_{t_1,t_2,u}}(y_1,y_2)dy_1dy_2=\lf(4\log \frac{1}{u}\rh)^2\lf(1-\frac{\log |t_1|}{2\log \frac{1}{u}}\rh)\lf(1-\frac{\log |t_2|}{2\log \frac{1}{u}}\rh).
	\end{equation*}

	\noindent Case 11 : $t_1, t_2\in (-1, -u^2)$. In this case
	\begin{equation*}
		I_{t_1,t_2,u} = I_{11}\cup I'_{11} \times I_{12}\cup I'_{12} 
\end{equation*}
	so that
	\begin{equation*}
	\et \frac{1}{|y_1y_2|}\chi_{I_{t_1,t_2,u}}(y_1,y_2)dy_1dy_2=\lf(4\log \frac{1}{u}\rh)^2\lf(1-\frac{\log |\frac{1}{t_1}|}{2\log \frac{1}{u}}\rh)\lf(1-\frac{\log |\frac{1}{t_2}|}{2\log \frac{1}{u}}\rh).
	\end{equation*}
	
	\noindent Case 12 : $t_1\in(-1, -u^2)$, $t_2\in (-\frac{1}{u^2},-1)$. In this case 
	\begin{equation*}
	I_{t_1,t_2,u} = I_{11}\cup I'_{11}  \times I_{13}\cup I'_{13}
	\end{equation*}
	so that
	\begin{equation*}
	\et \frac{1}{|y_1y_2|}\chi_{I_{t_1,t_2,u}}(y_1,y_2)dy_1dy_2=\lf(4\log \frac{1}{u}\rh)^2\lf(1-\frac{\log |\frac{1}{t_1}|}{2\log \frac{1}{u}}\rh)\lf(1-\frac{\log |{t_2}|}{2\log \frac{1}{u}}\rh).
	\end{equation*}
	
	\noindent Case 13 : $t_1\in(u^2, 1)$, $ t_2\in ( -1, -u^2)$. In this case, it can be worked out that 
	\begin{equation*}
	I_{t_1,t_2,u} = I_6\cup I'_6  \times I_{12}\cup I'_{12} 
	\end{equation*}
	and therefore,
	\begin{equation*}
\et \frac{1}{|y_1y_2|}\chi_{I_{t_1,t_2,u}}(y_1,y_2)dy_1dy_2=\lf(4\log \frac{1}{u}\rh)^2\lf(1-\frac{\log |\frac{1}{t_1}|}{2\log \frac{1}{u}}\rh)\lf(1-\frac{\log |\frac{1}{t_2}|}{2\log \frac{1}{u}}\rh).
	\end{equation*}
		
	\noindent Case 14 : $t_1\in (1,\frac{1}{u^2})$, $t_2\in (-1, -u^2)$. In this case
	\begin{equation*}
	I_{t_1,t_2,u} = I_8\cup I'_8 \times I_{12}\cup I'_{12} 
	\end{equation*}
	so that
	\begin{equation*}
\et \frac{1}{|y_1y_2|}\chi_{I_{t_1,t_2,u}}(y_1,y_2)dy_1dy_2=\lf(4\log \frac{1}{u}\rh)^2\lf(1-\frac{\log |{t_1}|}{2\log \frac{1}{u}}\rh)\lf(1-\frac{\log |\frac{1}{t_2}|}{2\log \frac{1}{u}}\rh).
	\end{equation*}

	\noindent Case 15 : $t_1\in(u^2, 1)$, $t_2\in (-\frac{1}{u^2}, -1)$. In this case
	\begin{equation*}
		I_{t_1,t_2,u} = I_6\cup I'_6 \times I_{13}\cup I'_{13} 
\end{equation*}
	so that
	\begin{equation*}
	\et \frac{1}{|y_1y_2|}\chi_{I_{t_1,t_2,u}}(y_1,y_2)dy_1dy_2=\lf(4\log \frac{1}{u}\rh)^2\lf(1-\frac{\log |\frac{1}{t_1}|}{2\log \frac{1}{u}}\rh)\lf(1-\frac{\log |{t_2}|}{2\log \frac{1}{u}}\rh).
	\end{equation*}
	
	\noindent Case 16 : $t_1\in(1, \frac{1}{u^2})$, $t_2\in (-\frac{1}{u^2},-1)$. In this case 
	\begin{equation*}
	I_{t_1,t_2,u} = I_8\cup I'_8  \times I_{13}\cup I'_{13}
	\end{equation*}
	so that
	\begin{equation*}
\et \frac{1}{|y_1y_2|}\chi_{I_{t_1,t_2,u}}(y_1,y_2)dy_1dy_2=\lf(4\log \frac{1}{u}\rh)^2\lf(1-\frac{\log |{t_1}|}{2\log \frac{1}{u}}\rh)\lf(1-\frac{\log |{t_2}|}{2\log \frac{1}{u}}\rh).
	\end{equation*}
	Combining the above information and taking $f$ and $g$ as $f_u$ and $g_u$ respectively in (\ref {h0}), we obtain that
	\begin{align}\label{2.t3}
	\mathcal{J}= &\et \et \frac{|\phi(t_1,t_2)|}{|t_1t_2|^{2\al+2}}h_u(t_1,t_2)dt_1dt_2\nonumber\\
	& \geq \et \frac{|\phi(t_1,t_2)|}{|t_1t_2|^{{2\al+2-\frac{1}{p}}}} \inf_{(y_1,y_2)\in \R^2}\lf(\frac {v(t_1y_1, t_2y_2)}{v(y_1,y_2)}\rh)^{\frac{1}{p}}\et \frac{1}{|y_1y_2|}\chi_{I_{t_1,t_2,u}}(y_1,y_2)dy_1dy_2 dt_1dt_2\nonumber\\
	& =\lf(4 \log \frac{1}{u}\rh)^2 \lf( \int_{-\frac{1}{u^2}}^{-u^2}\int_{-\frac{1}{u^2}}^{-u^2}+ \int_{u^2}^{\frac{1}{u^2}}\int_{u^2}^{\frac{1}{u^2}}+\int_{u^2}^{\frac{1}{u^2}}\int_{-\frac{1}{u^2}}^{-u^2}+\int_{-\frac{1}{u^2}}^{-u^2}\int_{u^2}^{\frac{1}{u^2}}\rh)\lf(\frac{|\phi(t_1,t_2)|}{|t_1t_2|^{{2\al+2}-\frac{1}{p}}}\rh.\times\nonumber\\
	& \quad \times  \inf_{(y_1,y_2)\in \R^2}\lf(\frac {v(t_1y_1, t_2y_2)}{v(y_1,y_2)}\rh)^{\frac{1}{p}}\lf.\lf(1-\frac{\xi(t_1)}{2\log \frac{1}{u}}\rh)\lf(1-\frac{\xi(t_2)}{2\log \frac{1}{u}}\rh)dt_1dt_2\rh),
	\end{align}
	where for $i=1,2$
		\begin{equation*}
	\xi(t_i)=
	\begin{cases}
	\log |\frac{1}{t_i}|,& {\rm if}\, t_i\in (u^2, 1]\cup(-1, -u^2] \\
	\log |t_i| , & {\rm if}\, t_i\in (1, 1/u^2]\cup (-1/u^2, -1).	
	\end{cases}	
	\end{equation*}
	Now,  by using the test functions $f_u,$ $g_u$ in (\ref{h2.1.1}) and using (\ref{h2.1.2}), (\ref{2.t3}), we get
	\begin{align*}
	&\lf( \int_{-\frac{1}{u^2}}^{-u^2}\int_{-\frac{1}{u^2}}^{-u^2}+ \int_{u^2}^{\frac{1}{u^2}}\int_{u^2}^{\frac{1}{u^2}}+\int_{u^2}^{\frac{1}{u^2}}\int_{-\frac{1}{u^2}}^{-u^2}+\int_{-\frac{1}{u^2}}^{-u^2}\int_{u^2}^{\frac{1}{u^2}}\rh) 
	\lf(\frac{|\phi(t_1,t_2)|}{|t_1t_2|^{{{2\al+2}-\frac{1}{p}}}} \inf_{(y_1,y_2)\in \R^2}\lf(\frac {v(t_1y_1, t_2y_2)}{v(y_1,y_2)}\rh)^{\frac{1}{p}}\rh.\times\\
	&\quad\times \lf.\lf(1-\frac{\xi(t_1)}{2\log \frac{1}{u}}\rh)\lf(1-\frac{\xi(t_2)}{2\log \frac{1}{u}}\rh)dt_1dt_2\rh)\leq \|H_{\al, \phi}\|.
	\end{align*}
By the Monotone Convergence Theorem and on taking $u\rightarrow 0$ we have 
	$$\et\frac{|\phi(t_1,t_2)|}{|t_1t_2|^{{{2\al+2}-\frac{1}{p}}}} \inf_{(y_1,y_2)\in \R^2}\lf(\frac {v(t_1y_1, t_2y_2)}{v(y_1,y_2)}\rh)^{\frac{1}{p}} dt_1dt_2\leq \|\sH_{\al,\phi}\|$$
	and we are done.
	\end{proof}
	
On the lines of Theorem \ref{t2.3}, a characterization of the boundedness of $\sH_{\al,\phi}:L^p_v(\R^2)\to L^p_v(\R^2)$  can be obtained. Precisely, we have the following

\begin{theorem}\label{t3.4}
	Let $1<p<\infty$, $v$ be weight function and $\phi\in L^1(\R^2)$. Let the following be satisfied for some constant $c>0$: 
	$$
	\sup_{(y_1,y_2)\in \R^2}\lf(\frac {v(t_1y_1, t_2y_2)}{v(y_1,y_2)}\rh)\le c \inf_{(y_1,y_2)\in \R^2}\lf(\frac {v(t_1y_1, t_2y_2)}{v(y_1,y_2)}\rh).
	$$ 
	Then the operator $\sH_{\al,\phi}:L^p_v(\R^2)\to L^p_v(\R^2)$ is  bounded if and only if $\sA_{\sup} <\infty$ and moreover the following estimates hold:
	\begin{equation*}
		{1\over c} \sA_{\sup} \le \|\sH_{\al,\phi}\|_{L^p_v(\R^2)\to L^p_v(\R^2)}\le \sA_{\sup}.
	\end{equation*}
\end{theorem} 

\begin{corollary}\label{t3.5} Let $1<p<\infty$, $v$ be weight function and $\phi\in L^1(\R^2)$. If there exists a function $h$ such that $v(x_1y_1,x_2y_2)=v(x_1,x_2)h(y_1,y_2)$, then $\sA_{\sup} = \sA_{\inf}$ and
	\begin{equation*}
		\|\sH_{\al,\phi}\|_{L^p_v(\R^2)\to L^p_v(\R^2)} = \ent \frac{|\phi(t_1,t_2)|}{{|t_1t_2|^{2\al+2-\frac{1}{p}}}}h^{1/p}(t_1,t_2)\, dt_1dt_2.
	\end{equation*}
	
\end{corollary}

\begin{remark}
For $\al=-1/2$, $\sH_{\al,\phi}$ reduces to the two dimensional Hausdorff operator

$$
\sH_{\phi}f(x_1,x_2)=\et \frac{|\phi(t_1,t_2)|}{|t_1t_2|}f\lf(\frac{x_1}{t_1},\frac{x_2}{t_2}\rh)dt_1dt_2
$$
which on replacement $\phi(s_1,s_2)=\frac{1}{s_1s_2}\psi\lf(\frac{1}{s_1},\frac{1}{s_2}\rh)$ becomes equivalent to 
$$
\mathscr G_\psi g(x_1,x_2)=\et \psi\lf(\frac{t_1}{x_1},\frac{t_2}{x_2}\rh) g(t_1,t_2)\,dt_1dt_2.
$$
The $L^p(\R^+\times \R^+)$-boundedness of $\mathscr G_\psi$ (consequently of $\mathscr H_\psi$) was proved in \cite{jain}. Moreover, if we take
$$
\phi(t_1,t_2)=\frac{1}{t_1t_2}\chi_{(1,\infty)}(t_1) \chi_{(1,\infty)}(t_2)
$$
then $\mathscr H_\psi f$ becomes the two-dimensional Hardy operator \cite{sawyer}
$$
H_2 f(x_1,x_2)=\frac{1}{x_1x_2}\int_0^{x_1} \int_0^{x_2} f(t_1,t_2)\,dt_1dt_2.
$$
\end{remark}

\section{Some generalizations}

In this section, we shall prove generalizations of some of the theorems proved in the previous section. To begin with, the following theorem is a two-weight generalization of Theorem \ref{t2.1}	

\begin{theorem}\label{t4.1}
	Let $1<p<\infty$, $v,w$ be weight functions and $\phi\in L^1(\R)$ be such that $$\displaystyle \ent \frac{|\phi(t)|}{{|t|^{{2\al+2}-\frac{1}{p}}}}\lf(\sup_{y\in\R}\frac{v(ty)}{w(y)}\rh)^{1/p} dt =:B_{\sup}<\infty.$$ Then the operator $H_{\al,\phi}:L^p_w(\R)\to L^p_v(\R)$ is bounded  and 
	$$\|H_{ \al,\phi}f\|_{L^p_v(\R)}\leq B_{\sup} \|f\|_{L^{p}_{w}(\R)}.
	$$
\end{theorem}
\begin{proof}
It follows on the similar lines as that of the proof of Theorem \ref{t2.1} by replacing $vv^{-1}$ by $ww^{-1}$	in \eqref{enew1}.
\end{proof}

 Theorem \ref{t4.1} has the following version for two indices $p,q$:
 
 \begin{theorem}\label{t4.2}
 	Let $1<q<p<\infty$, $v,w$ be weight functions and $\phi\in L^1(\R)$ be such that $$\displaystyle \ent \frac{|\phi(t)|}{{|t|^{{2\al+2}-\frac{1}{p}}}}\lf(\ent\frac{[v(ty)]^\frac{p}{p-q}}{[w(y)]^\frac{q}{p-q}}dy\rh)^\frac{p-q}{pq} dt =:D_{\sup}<\infty.$$ Then the operator $H_{\al,\phi}:L^p_w(\R)\to L^q_v(\R)$ is bounded  and 
 	$$\|H_{ \al,\phi}f\|_{L^q_v(\R)}\leq D_{\sup} \|f\|_{L^{p}_{w}(\R)}.
 	$$
 \end{theorem}
 
 \begin{proof}
 		Let $f \in L^p_v(\R)$ then by using generalised Minkowski inequality, change of variables, H\"older's inequality for $p/q>1$ and , we have
	\begin{align*}
	\|H_{ \al,\phi}f\|_{L^q_v(\R)}
	& =\lf(\ent |H_{ \al,\phi} f(x)|^q v(x)dx \rh)^\frac{1}{q} \\
	& = \lf(\ent\lf|\ent \frac{|\phi(t)|}{|t|^{2\al+2}}f\lf(\frac{x}{t}\rh) dt\rh|^q v(x)dx \rh)^\frac{1}{q} \\
	& \leq \lf(\ent \lf(\ent \frac{|\phi(t)|}{|t|^{2\al+2}}\lf|f\lf(\frac{x}{t}\rh)\rh|v^{\frac{1}{q}}(x)dt\rh)^q dx \rh)^\frac{1}{q}\\
	& \leq  \ent \frac{|\phi(t)|}{|t|^{2\al+2}}|t|^{\frac{1}{q}} \lf(\ent \lf|f(y)\rh|^q v(yt)dy\rh)^\frac{1}{q} dt\\
	& =  \ent \frac{|\phi(t)|}{|t|^{{2\al+2}-\frac{1}{q}}}\lf(\ent \lf|f(y)\rh|^q w^{\frac{q}{p}}(y)w^{-\frac{q}{p}}(y) v(yt)dy\rh)^{\frac{1}{q}} dt\\
	&\leq \lf(\ent \frac{|\phi(t)|}{|t|^{{2\al+2}-\frac{1}{q}}} \lf(\ent\frac{[v(ty)]^\frac{p}{p-q}}{[w(y)]^\frac{q}{p-q}}dt\rh)^\frac{p-q}{pq} dt \rh)\lf(\ent |f(y)|^p w(y) dy \rh)^{\frac{1}{p}}\\
	& = D_{\sup} \|f\|_{L^p_w(\R)}
	\end{align*}
and the assertion follows.

 \end{proof}
 
 Two dimensional versions of Theorems \ref{t4.1} and \ref{t4.2} can also be proved. We only state the results.
 
 \begin{theorem}
	Let $1<p<\infty$, $v,w$ be weight functions and $\phi\in L^1(\R^2)$ be such that $$\displaystyle \et \frac{|\phi(t_1, t_2)|}{{|t_1t_2|^{{2\al+2}-\frac{1}{p}}}}\lf(\sup_{(y_1, y_2)\in\R^2}\frac{v(t_1y_1, t_2y_2)}{w(y_1,y_2)}\rh)^{1/p} dt_1dt_2 =:\mathscr B_{\sup}<\infty.$$ Then the operator $\sH_{\al,\phi}:L^p_w(\R^2)\to L^p_v(\R^2)$ is bounded  and 
	$$\|\sH_{ \al,\phi}f\|_{L^p_v(\R^2)}\leq {\mathscr B}_{\sup} \|f\|_{L^{p}_{w}(\R^2)}.
	$$
 	
 \end{theorem}
 
 \begin{theorem}
 	 	Let $1<q<p<\infty$, $v,w$ be weight functions and $\phi\in L^1(\R^2)$ be such that $$\displaystyle \et \frac{|\phi(t_1,t_2)|}{{|t_1t_2|^{{2\al+2}-\frac{1}{p}}}}\lf(\et \frac{[v(t_1y_1, t_2y_2)]^\frac{p}{p-q}}{[w(y_1,y_2)]^\frac{q}{p-q}}\rh)^\frac{p-q}{pq} dt_1dt_2 =:\mathscr D_{\sup}<\infty.$$ Then the operator $\mathscr H_{\al,\phi}:L^p_w(\R^2)\to L^q_v(\R^2)$ is bounded  and 
 	$$\|\mathscr H_{ \al,\phi}f\|_{L^q_v(\R^2)}\leq \mathscr D_{\sup} \|f\|_{L^{p}_{w}(\R^2)}.
 	$$
 \end{theorem}
 
 \noindent {\it Acknowledgement.} The third author acknowledges the MATRICS Research Grant No. MTR/2017/000126 of SERB, Department of Science and Technology, India.

\end{document}